%% file: manuscript.tex
\documentclass{article}

\usepackage{graphicx}
\usepackage{amssymb}
\usepackage{amsmath}
\usepackage{enumitem}
\usepackage{hyperref}

\newtheorem{definition}{Definition}

\newtheorem{theorem}{Theorem}
\newtheorem{lemma}{Lemma}

\newcommand{\eproof}{{\hfill $\square$ \bigskip}}
\newcommand{\beginproof}{{\it Proof.\ }}

\author{
	Franti\v{s}ek Kardo\v{s}
	\thanks{LaBRI, CNRS, University of Bordeaux, Talence, F-33405, France; \texttt{fkardos@labri.fr}.}
\and 
	Ma\'{u}\v{s} Matok
	\thanks{Department of Computer Science, Faculty of Mathematics, Physics and Informatics, Comenius University, Mlynsk\'{a} Dolina, 842 48 Bratislava, Slovakia; \texttt{matus.matok@fmph.uniba.sk}.}
}
\title{Total coloring of (sub)cubic Halin graphs}
\begin{document}
\maketitle
\begin{abstract}
Total coloring of a graph is a coloring of its vertices and edges such that adjacent or incident elements receive distinct colors.
Total coloring conjecture (stipulating that the total chromatic number of a graph $G$ is at most $\Delta(G)+2$) is known to be true for subcubic graphs -- five colors are always enough. However, deciding whether a total coloring with only four colors exists remains a difficult problem, even in the class of bipartite cubic graphs. We solve the problem completely for cubic and subcubic Halin graphs, proving that there are only finitely many such graphs requiring five colors.
\end{abstract}
\section{Introduction}
\label{sec:introduction}

\input{article/1-introduction}

\section{Results for cubic Halin graphs}
\label{sec:proof}
\input{article/2-cubic}

\section{Results for subcubic Halin graphs}
\label{sec:subcubic}
\input{article/3-subcubic}

\section{Concluding remarks}
\label{sec:conclude}

\input{article/4-conclusion}


\bibliographystyle{habbrv}
\bibliography{sources} 

\end{document}

%% file: article/1-introduction.tex
Total coloring was introduced very soon after the edge coloring as a natural generalization of the latter. A \emph{total coloring} is an assignment of colors to vertices and edges of a given graph such that adjacent or incident elements receive distinct colors; the smallest number of colors in a total coloring of a given graph $G$ is called the \emph{total chromatic number} of $G$ and is denoted by $\chi''(G)$.

Vizing \cite{vizing1964estimate} (and independently Behzad \cite{behzad1965graphs,behzad1969total}) conjectured that every simple graph is totally $(\Delta(G)+2)$-colorable, as an analog to his upper bound of the form $\Delta(G)+1$ for the chromatic index of any (simple) graph \cite{vizing1964estimate} (see also \cite{berge1991short}). This claim has nowadays been known as the Total Coloring Conjecture (TTC). 

The TTC was verified for $\Delta = 3$ in 1971 by Rosenfeld \cite{rosenfeld1971total} and independently by Vijayaditya \cite{vijayaditya1971total} (see \cite{FENG2013-TC-subcubic} a new short proof). For $\Delta = 4$ and $\Delta = 5$ it was proven by Kostochka \cite{KOSTOCHKA1977totalDelta4, KOSTOCHKA1996totalDelta5}. In the general case, the best upper bound known so far is of the form $\Delta + 10^{26}$ as shown by Molloy and Reed \cite{Molloy1998totalGeneralBound}, see also \cite{Dalal2025total}. See the survey \cite{Geetha02092023} for more general results on total coloring.

For $\Delta = 3$, i.e., the class of subcubic graphs, there are only two possible values of the total chromatic number, namely 4 and 5. Subcubic graphs that attain the lower bound of 4 are said to be of \emph{Type 1} while the rest are said to be of \emph{Type 2}. Determining $\chi''$ is an NP-complete problem in general \cite{SanchezArroyo1989NP-total-general}, and remains so even for cubic bipartite graphs \cite{McDiarmid1994NP-total-bipartite}. 
On the other hand, for several subclasses of (sub)cubic graphs it is either known that they are totally 4-colorable (like truncated cubic graphs -- those where every vertex is in exactly one triangle \cite{anthonymuthu2022total}), or deciding whether this is true is polynomial (like graphs of bounded treewidth \cite{Isobe1999PolynomialTotal-k-trees}). Almost all generalized Petersen graphs are of Type 1 \cite{Dantas20161471}: for each integer $k\ge 2$, there exists an integer $N(k)$ such that, for any $n\ge N(k)$ the generalized Petersen graph $G(n,k)$ has total chromatic number 4.

Adjacent-vertex-distinguishing (AVD) coloring, introduced by Zhang et al. \cite{Zhang2005AVD-intro} and denoted by $\chi'_{avd}$, coincides with total coloring in cubic graphs when using four colors. Indeed, in an AVD-coloring, the necessity of distinct sets of colors on incident edges of two adjacent vertices implies that each set lacks one color from the color set, which corresponds to the color of that vertex in a total coloring. Similarly as for total coloring, Balister et al. \cite{LEHEL2007-AVD-subcubic} showed that $\chi'_{avd}(G) \leq 5$ for graphs $G$ with $\Delta = 3$.

A graph is \emph{Halin} if it can be obtained from a plane embedding of a tree with at least three leaves by adding a cycle passing through all the leaves with no crossing edges. Usually, the tree is required not to contain vertices of degree two. 
Every Halin graph is connected and planar. 
In addition to $K_4$, only three totally non-4-colorable cubic Halin graphs have been known, see Figure \ref{fig:halin}. The first two of them were identified by Huang and Chen \cite{Huang2023-halin}, who proved that these graphs are the only Type 2 graphs within the subclass of cubic Halin graphs whose tree is a caterpillar (it becomes a path when leaves are ignored). Lužar \cite{borut-personal2024} found a third one using computer assistance.

\begin{figure}
    \centering
    \includegraphics{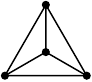}
    $\qquad$
    \includegraphics{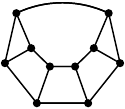}
    $\qquad$
    \includegraphics{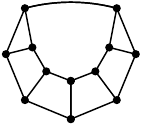}
    $\qquad$
    \includegraphics{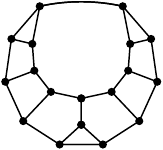}
    \caption{The only four known cubic Halin graphs with $\chi'' = 5$, referred to as $K_4$, $H_4$, $H_5$, and $H_8$ (from left to right).}
    \label{fig:halin}
\end{figure}

In this article we fully characterize Halin graphs of Type 2. 
More precisely, using computer-assisted techniques, we prove that the four graphs mentioned above are the only four cubic Halin graphs of Type 2. Additionally, we extend the characterization to subcubic Halin graphs with regards to total and AVD coloring. 

%% file: article/2-cubic.tex
\begin{theorem}\label{thm:total-cubic}
    Let $H$ be a cubic Halin graph other than $K_4, H_4, H_5$ and $H_8$. Then 
    \[
    \chi''(H) = 4.
    \]
\end{theorem}

Before we prove our main result, let us introduce some auxiliary notions and notations.

\subsection{Halin tripoles and their (pre)colorings}

Let $(X,Y)$ be a partition of the vertex set of a graph $G$. Let $E(X,Y)$ denote the set of edges with one end-vertex in both $X$ and $Y$. Then $G[X]\cup E(X,Y)$ is a \emph{multipole} -- a graph in which some edges are only incident to one vertex (sometimes also called semi-edges or partial edges).
In this article, we will only consider planar cubic or subcubic \emph{tripoles} -- multipoles with three semi-edges -- with such an embedding that all the semi-edges are incident with the outer face.

Let $H$ be a Halin graph where $T$ is its internal tree. The vertices of $H$ corresponding to the leaves of $T$ are called \emph{peripheral vertices}. Vertices that are not peripheral are \emph{spanning vertices}. Edges joining 2 peripheral vertices are \emph{peripheral edges}. Edges that are not peripheral are \emph{spanning edges}. Observe that all the peripheral vertices are of degree three. 

Let $H$ be a Halin graph where $T$ is its internal free. Every spanning edge separates $T$ into two subtrees, let $X$ and $Y$ be their vertex sets. Then $E_H(X,Y)$ consists of three edges, one spanning and two peripheral. Any tripole of the form $H[X]\cup E(X,Y)$ where $(X,Y)$ is the partition of $V(G)$ defined by a spanning edge is called a \emph{Halin} tripole. The semi-edge corresponding to the spanning edge of $v$ will be called the \emph{root} semi-edge and the vertex incident with it in the tripole will be called \emph{root} vertex. In particular, we will denote $T^H_v$ the tripole $H[X]\cup E(X,Y)$ with $Y = \{v\}$ and $X = V(H)\setminus X$, obtained from $H$ by removing the vertex $v$.  

A tripole with $|X|=1$ is \emph{trivial}. It is the only tripole where the root vertex is peripheral.
The \emph{rank} $r(T)$ of a Halin tripole $T$ is the number of spanning vertices of $T$. Clearly, $r(T) = 0$ if and only if $T$ is trivial.

Let $T$ be a Halin tripole, let $r$ be its root edge, let $x$ and $y$ be its other two semi-edges, and let $r^*$ ($x^*$ and $y^*$) be the vertex $r$ ($x$ and $y$, respectively) is incident with. The sextuple $(r,r^*,x,x^*,y,y^*)$ is called the \emph{extended boundary} of $T$ and denoted by $\partial T$. Note that $r^*=x^*=y^*$ for a trivial tripole.

We need to redefine the notion of isomorphism for Halin tripoles. Let $T$ and $T'$ be two Halin tripoles with extended boundaries $\partial T$ and $\partial T'$. A graph isomorphism $\varphi : V(T) \to V(T')$ is a Halin tripole isomorphism if it is an orientation-preserving isomorphism from $T$ to $T'$ (as plane maps), such that $\varphi(\partial T) = \partial T'$. For instance, the tripoles $T_1^{H_5}$ and $T_2^{H_5}$ depicted in Figure \ref{fig:ambiguous-tripoles} are not considered isomorphic, despite the existence of an orientation-preserving map between them -- it does not respect the order of the semi-edges.

Observe that every total 4-coloring $\varphi : V(T) \cup E(T) \to \{0,1,2,3\}$ of a Halin tripole $T$ defines a 4-coloring of its extended boundary: 
It is the sextuple 
$$
\varphi|_{\partial T} = (\varphi(r),\varphi(r^*),\varphi(x),\varphi(x^*),\varphi(y),\varphi(y^*))
$$ 
of colors satisfying $\varphi(r)\ne\varphi(r^*)$, $\varphi(x)\ne\varphi(x^*)$, and $\varphi(y)\ne\varphi(y^*)$. Let 
$$
{S} = \{(a,b,c,d,e,f) \in \{0,1,2,3\}^6 \mid a\ne b, c\ne d, e\ne f\}
$$
denote the set of all potential precolorings of the extended boundary of a (Halin) tripole. For a given Halin tripole $T$, a precoloring $s\in {S}$ is \emph{extendable} if there exists a total 4-coloring $\varphi$ of $T$ such that $\varphi |_{\partial T} = s$. The set of all the extendable precolorings of $\partial T$ is called the \emph{palette} of $T$ and is denoted by $P(T)$. More generally, any subset of ${S}$ is a \emph{palette}. We say that a palette $P$ is \emph{realizable} if there exists a Halin tripole $T$ such that $P(T) = P$. 

Let $s = (a,b,c,d,e,f) \in {S}$. We say that $s$ is \emph{completable} if 
\begin{enumerate}[label=(\roman*)]
    \item $|\{a,c,e\}| = 3$, and
    \item $|\{a,b,c,d,e,f\}| = 3$.
\end{enumerate}

\begin{lemma}
    \label{lemma:halin-tripole-palette}
    Let $H$ be a cubic Halin graph and let $v$ be a peripheral vertex of $H$. Then $H$ is totally $4$-colorable if and only if $P(T^H_v)$ contains a completable extendable precoloring.
\end{lemma}

\beginproof
Assume $H$ is totally $4$-colorable. Then there exists a total coloring $\varphi$ of $H$. Then
\begin{enumerate}[label=(\roman*)]
    \item the three edges around $v$ are all colored distinct, and \label{arg:1}
    \item there are precisely $3$ distinct colors (argument \ref{arg:1} being the lower bound) on the edges and vertices incident with and adjacent to $v$.
\end{enumerate}
These conditions translate to the conditions in the definition of a completable palette immediately.

On the other hand, assume that $T^H_v$ has a completable extendable precoloring. By definition, it extends to a total coloring $\varphi$ of $T^H_v$. We may extend it to a coloring of $H$:
Since $|\{a,c,e\}| = 3$, there is no conflict among the edges incident with $v$. Furthermore, the condition $|\{a,b,c,d,e,f\}| = 3$ implies that there exists a color $i \notin \{a,b,c,d,e,f\}$ which can be used to color the vertex $v$. \eproof

\smallskip
Actually, the palette of a given Halin tripole can be efficiently computed using a dynamic programming approach.

Let $T_1$ and $T_2$ be two (not necessarily non-trivial) Halin tripoles with extended boundaries $(r_1,r_1^*,x_1,x_1^*,y_1,y_1^*)$ and $(r_2,r_2^*,x_2,x_2^*,y_2,y_2^*)$. Then the tripole defined by
$$
\begin{aligned}
    V(T) &= V(T_1)\cup V(T_2) \cup \{r^*\} \\
    E(T) &= (E(T_1)\setminus \{r_1,y_1\}) \cup (E(T_2)\setminus\{r_2,x_2\})\cup \{y_1^*x_2^*,r_1^*r^*,r_2^*r^*,r\}
\end{aligned}
$$
is a Halin tripole with extended boundary $(r,r^*,x_1,x^*_1,y_2,y^*_2)$; we will call it the \emph{composition} of $T_1$ and $T_2$ and we will write simply $T = T_1 + T_2$. It is straightforward to check the following statement.

\begin{lemma}
    Let $T$ be a non-trivial Halin tripole of rank $k$ rooted at the vertex $r$. Let $r_1$ and $r_2$ be the neighbors of $r$ in $T$. Then $T = T_1+T_2$, where $T_1$ and $T_2$ are Halin tripoles rooted at $r_1$ and $r_2$ of ranks $k_1$ and $k_2$ with $k_1+k_2+1=k$.
\end{lemma}

\begin{figure}
    \centering
    \includegraphics{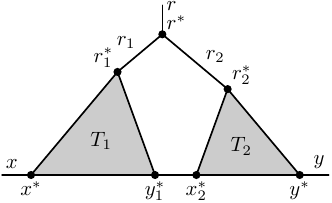}
    \caption{A Halin tripole $T$ decomposes into two Halin tripoles $T_1$ and $T_2$. }
    \label{fig:halin-tripole-decomp}
\end{figure}

Observe that the binary operation $+$ on Halin tripoles is neither commutative nor associative.

It is easy to observe that if $T=T_1+T_2$, then any total 4-coloring of $T$ induces total 4-colorings of both $T_1$ and $T_2$. On the other hand, not every pair or total 4-colorings of $T_1$ and $T_2$ can be combined into a total 4-coloring of $T$. To precise when this is the case, we introduce the following definition.

\begin{definition}
Let $s_1 = (a_1,b_1,c_1,d_1,e_1,f_1), s_2=(a_2,b_2,c_2,d_2,e_2,f_2) \in S$ be two sextuples.  We say that $s_1$ and $s_2$ are \emph{composable} if
\begin{enumerate}[label=(\roman*)]
    \item $e_1 = c_2$, \label{def:compos-1}
    \item $|\{e_1,f_1,c_2,d_2\}| = 3$, \label{def:compos-2}
    \item $a_1 \neq a_2$, and \label{def:compos-3}
    \item $|\{a_1,b_1,a_2,b_2\}| \leq 3$. \label{def:compos-4}
\end{enumerate}

Let $P_1$ and $P_2$ be two palettes.
Let $s_1=(a_1,b_1,c_1,d_1,e_1,f_1) \in P_1$ and $s_2=(a_2,b_2,c_2,d_2,e_2,f_2) \in P_2$ be a pair of composable sextuples.  We say that $s_1$ and $s_2$ \emph{yield} a sextuple $s=(a,b,c,d,e,f)$ if
\begin{enumerate}[label=(\Roman*)]
    \item $c=c_1$, $d=d_1$, $e=e_2$, $f=f_2$,\label{def:yield-1}
    \item $b \notin\{a_1,b_1,a_2,b_2\}$,\label{def:yield-2}
    \item $a\notin \{b,a_1,a_2\}$.\label{def:yield-3}
\end{enumerate}
Let
$$s_1\oplus s_2 := \{s\in {S} \mid \textrm{$s_1$ and $s_2$ yield $s$} \}$$
if $s_1$ and $s_2$ are composable, otherwise let $s_1\oplus s_2 := \emptyset$.
Finally, let
$$
P_1\oplus P_2 := \bigcup_{\substack{s_1\in P_1\\ s_2\in P_2}} s_1\oplus s_2.
$$
\end{definition}

Notice that we cannot define the composition of composable sextuples as a binary operation, since it might not be unique if $|\{a_1,b_1,a_2,b_2\}|=2$.

\begin{lemma}
    Let $T$ be a non-trivial Halin tripole such that $T = T_1 + T_2$. 
    Then $P(T) = P(T_1)\oplus P(T_2)$.
\label{lemma:decompose}
\end{lemma}

\beginproof
Let $s = (a,b,c,d,e,f) \in {S}$. We will prove that 
$s\in P(T)$ if and only if there exist $(a_1,b_1,c_1,d_1,e_1,f_1) \in P(T_1)$ and $(a_2,b_2,c_2,d_2,e_2,f_2) \in P(T_2)$ which are composable and which yield $s$.

Let $\partial T = (r,r^*,x,x^*,y,y^*)$, let $\partial T_1 = (r_1,r^*_1,x_1,x^*_1,y_1,y^*_1)$, and let $\partial T_2 = (r_2,r^*_2,x_2,x^*_2,y_2,y^*_2)$. Then $x=x_1$, $x^*=x^*_1$, $y=y_2$, $y^*=y^*_2$, the vertices $y^*_1$ and $x^*_2$ are adjacent in $T$ and the semi-edges $x_2$ and $y_1$ correspond to the same edge $x^*_2y^*_1$ in $T$.

Assume $s\in P(T)$. Then $s$ is an extendable precoloring of $\partial T$, and so there exists a total 4-coloring $\varphi$ of $T$, which induces total 4-colorings $\varphi_1$ and $\varphi_2$ of $T_1$ and $T_2$, respectively. Let 
$s_1 = \varphi_1|_{\partial T_1} = (a_1,b_1,c_1,d_1,e_1,f_1)$ and let
$s_2 = \varphi_2|_{\partial T_2} = (a_2,b_2,c_2,d_2,e_2,f_2)$. Clearly, $s_1\in P(T_1)$ and $s_2\in P(T_2)$. We have 
$$
e_1 = \varphi_1(y_1) = \varphi (x^*_2y^*_1) = \varphi_2(x_2) = c_2
$$
and, moreover, the colors $f_1 = \varphi(y^*_1)$ and $d_2 = \varphi(x^*_2)$ being the colors of the endvertices of $x^*_2y^*_1$, they are different from each other and they are both distinct from $e_1=c_2 = \varphi(x^*_2y^*_1)$, and so $s_1$ and $s_2$ satisfy the conditions \ref{def:compos-1} and \ref{def:compos-2} in the definition of a composable pair. Moreover, they satisfy \ref{def:compos-3} since $a_1$ and $a_2$ are the colors of two distinct edges incident with $r^*$, and they satisfy \ref{def:compos-4} as the colors $a_1$, $b_1$, $a_2$, $b_2$ are all distinct from $\varphi(r^*)$. Hence, $s_1$ and $s_2$ are composable. It is a routine check that $s_1$ and $s_2$ yield $s$: The colors of identical elements (edges or vertices) coincide, and so \ref{def:yield-1} is true; the conditions \ref{def:yield-2} and \ref{def:yield-3} follow directly from the definition of a total coloring.

On the other hand, assume that there exist $s_1\in P(T_1)$ and $s_2\in P(T_2)$ that are composable and they yield $s$. Then $s_1$ ($s_2$) extends to a total 4-coloring of $T_1$ ($T_2$, respectively). We claim that we can combine the two colorings into a total 4-coloring $\varphi$ of $T$. Indeed, there is no conflict around the edge $x^*_2y^*_1$ thanks to conditions \ref{def:compos-1} and \ref{def:compos-2}, and there is no conflict between $r^*_1r^*$ and $r^*_2r^*$ by \ref{def:compos-3}. 
Thanks to the conditions \ref{def:compos-4} and \ref{def:yield-2}, the color $b$ can be assigned to $r^*$, and by \ref{def:yield-3}, the color $a$ can be assigned to $r$. The coloring $\varphi$ clearly satisfies \ref{def:yield-1}.  \eproof

Observe that since the size of a palette is bounded by a constant, the composition of two palettes can always be computed in a constant time.

\begin{theorem}
    Let $H$ be a cubic Halin graph. An instance of total 4-coloring in $H$ can be found in linear time (if it exists).
\end{theorem}

\beginproof
We can arbitrarily choose a peripheral vertex $v$ of $H$ and calculate $P(T^H_v)$ by applying Lemma \ref{lemma:decompose} recursively, memorizing the palettes of all the (linearly many) sub-tripoles along the way. According to Lemma \ref{lemma:halin-tripole-palette}, it suffices to check whether $P(T^H_v)$ contains a sextuple that is completable (which can be done in constant time). If this is the case, then we fix such a completable precoloring $s$, and extend it into a total 4-coloring of $T^*_v$ by the following algorithm: For a given non-trivial Halin tripole $T$ and its extendable precoloring $s$, consider the tripoles $T_1$ and $T_2$ such that $T = T_1 + T_2$ (whose palettes have already been calculated) and pick precolorings $s_1\in P(T_1)$ and $s_2\in P(T_2)$ such that $s_1$ and $s_2$ yield $s$; set the colors of the root edges and root vertices of $T_1$ and $T_2$ to be those given by $s_1$ and $s_2$; continue recursively for $T_1$ and $T_2$. For a trivial Halin tripole, there is nothing to do, the colors of all its elements (vertex and edges) are already defined in any extendable precoloring. Now that we have a total 4-coloring of $T^H_v$, it suffices to color the vertex $v$, which is possible due to $s$ being completable.
\eproof 

\subsection{Palettes of Halin tripoles}
\label{subsec:palettes}

Let $P_0$ be the palette of the Halin tripole of rank 0. A set of palettes $\mathcal{P}$ is \emph{closed under composition} if for any two palettes $P_x, P_y \in \mathcal{P}$ we have $P_x \oplus P_y \in \mathcal{P}$. Let $\mathcal{P}^+$ be the inclusion-wise minimal set of palettes containing $P_0$ closed under composition. Clearly, $\mathcal{P}^+$ is the set of all realizable palettes. 

We define the \emph{rank} of a realizable palette $P$ as the smallest $k$ such that $P = P(T)$ for a Halin tripole $T$ of rank $k$. For instance, 
$$
P_0=\{(b,a,c,a,d,a)\in \mathcal{S} \mid \{a,b,c,d\}=\{0,1,2,3\}\}$$
is a palette of rank 0, since it is realized by the trivial tripole. In fact, this is the only palette of rank 0.

Since there are only finitely many palettes, their ranks are bounded by a constant. Let $K$ be a maximum rank of a palette, and for $k=0,1,\dots, K$, let $\mathcal{P}_k$ be the set of palettes of rank $k$ and let $\mathcal{S}_k$ be the set of palettes of rank at most $k$.
Then $\{\mathcal{P}_0,\dots,\mathcal{P}_K\}$ is a partition of a set of realizable palettes $\mathcal{P}^+$, that we will call the \emph{stratification} of $\mathcal{P}^+$; its element $\mathcal{P}_i$ is said to be the \emph{stratum} $i$. The next statement follows directly from the definitions.

\begin{lemma}\label{lemma:recfor}
$P \in \mathcal{S}_k$ if and only if 
$$P\in  \mathcal{S}_{k-1} ~ \textrm{ or } ~ \exists ~ 0\le i \le k-1, ~\exists~ P_1\in \mathcal{P}_{i}, ~\exists P_2\in \mathcal{P}_{k-1-i} \textrm{ such that } P = P_1\oplus P_2.
$$
\end{lemma}

We can apply the formula from Lemma \ref{lemma:recfor} iteratively until $\mathcal{S}_k$ stabilizes. Since $\mathcal{P}_k = \mathcal{S}_k \setminus \mathcal{S}_{k-1}$, we can compute the complete stratification in this way.

 A palette is \emph{incompletable} if it does not contain any completable sextuple.

\begin{table}[ht]
    \centering
    \begin{tabular}{|c|c|c|c|}\hline
         Stratum & All & UD & Incompletable\\\hline
         $\mathcal{P}_0$ & 1 & 1 & 1 \\
         $\mathcal{P}_1$ & 1 & 1 & 1\\
         $\mathcal{P}_2$ & 2 & 2 & 0\\
         $\mathcal{P}_3$ & 5 & 5 & 0\\
         $\mathcal{P}_4$ & 14 & 14 & 6\\
         $\mathcal{P}_5$ & 38 & 27 & 7\\
         $\mathcal{P}_6$ & 83 & 62 & 0\\
         $\mathcal{P}_7$ & 165 & 116 & 0\\
         $\mathcal{P}_8$ & 239 & 152 & 7\\
         $\mathcal{P}_9$ & 207 & 114 & 0\\
         $\mathcal{P}_{10}$ & 201 & 108 & 0\\
         $\mathcal{P}_{11}$ & 137 & 80 & 0\\
         $\mathcal{P}_{12}$ & 87 & 50 & 0\\
         $\mathcal{P}_{13}$ & 20 & 6 & 0\\
         $\mathcal{P}_{14}$ & 6  & 4 & 0\\
         $\mathcal{P}_{15}$ & 8 & 6 & 0\\\hline
         $\mathcal{P}_{i}, i \ge 16$ & 0 & 0 & 0 \\\hline
         total & 1214 & 748 & 22 \\\hline
    \end{tabular}
    \caption{The strata of $\mathcal{P}^+$, with total number of palettes per stratum (first column), number of uniquely decomposable (UD) palettes per stratum (second column), and the number of incompletable palettes per stratum (last column).}
    \label{tab:stratification-cubic}
\end{table}

\begin{lemma}\label{lemma:p-plus-size}
    The size of $\mathcal{P}^+$ is 1214. Its stratification consists of $16$ non-empty strata. There are 22 incompletable palettes in $\mathcal{P}^+$.
\end{lemma}

\beginproof
We have implemented a program to execute the formula from Lemma \ref{lemma:recfor} until no new palettes are created.
\eproof

The results, i. e., sizes of non-empty strata,  can be seen in Table \ref{tab:stratification-cubic}.

The graphs $K_4$, $H_4$, $H_5$, and $H_8$ give rise to 1, 6, 7, and 10 cubic Halin tripoles of the form $T^H_v$, respectively, all of them having incompletable palettes. Moreover, the palette $P_0$ of the trivial cubic Halin tripole is also incompletable. Let us call these tripoles \emph{bad}.
There are 25 bad tripoles but only 22 incompletable palettes. To explain the difference, it suffices to realize that there are three pairs of bad tripoles with identical palettes, denoted by $T^{H_5}_1$, $T^{H_5}_2$, $T^{H_8}_1$, $T^{H_8}_2$, $T^{H_8}_3$, $T^{H_8}_4$, see Figure \ref{fig:ambiguous-tripoles}. Let us denote their palettes by  

$$
\begin{aligned}
P_A = P(T^{H_5}_1) &= P(T^{H_8}_1), \\   
P_B = P(T^{H_5}_2) &= P(T^{H_8}_2), \\
P_C = P(T^{H_8}_3) &= P(T^{H_8}_4).
\end{aligned}
$$

\begin{figure}[ht]
    \centering
    \begin{tabular}{c c c}
         \includegraphics[]{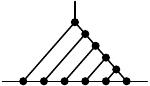}& \includegraphics[]{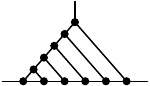} & \includegraphics[]{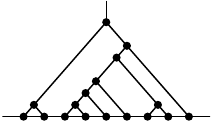}  \\
         $T^{H_5}_1$&$T^{H_5}_2$ & $T^{H_8}_3$\\
         & & \\
         \includegraphics[]{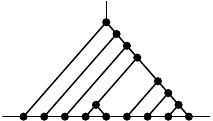}& \includegraphics[]{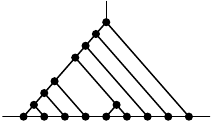} & \includegraphics[]{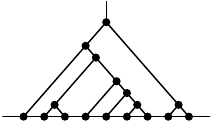} \\
         $T^{H_8}_1$&$T^{H_8}_2$ & $T^{H_8}_4$
    \end{tabular}
    \caption{The six bad tripoles without unique incompletable palettes.}
    \label{fig:ambiguous-tripoles}
\end{figure}

Let $T^0, T^1,T^4_1,T^4_2,T^6_1, T^6_2, T^7_1,$ and $T^7_2$ be tripoles such that

$$
\begin{aligned}
    T^{H_5}_1 &= T^0 + T^4_1, \qquad T^{H_5}_2 = T^4_2 + T^0, \qquad T^{H_8}_3 = T^1 + T^6_1, \\
    T^{H_8}_1 &= T^0 + T^7_1, \qquad T^{H_8}_2 = T^7_2 + T^0, \qquad T^{H_8}_4 = T^6_2 + T^1.
\end{aligned}
$$

Note that $T^0$ is the trivial tripole and $T^1$ is the only tripole of rank 1. It was checked by a computer that
$$
\begin{aligned}
    P_A &= P(T^0) + P(T^4_1) = P(T^0) + P(T^7_1), \\
    P_B &= P(T^4_2) + P(T^0) = P(T^7_2) + P(T^0), \\
    P_C &= P(T^1) + P(T^6_1) = P(T^6_2) + P(T^1)
\end{aligned}
$$
are the only possible decompositions of the palettes $P_A$, $P_B$, and $P_C$, respectively.

\begin{lemma}
\label{lemma:trivial}
Let $T$ be a realizable tripole, and let $T^0$ be the trivial tripole. If $P(T) = P(T^0)$ then $T \cong T^0.$ 
\end{lemma}

\beginproof It was checked by a computer that for every pair of realizable palettes $P_1$ and $P_2$, $P_1\oplus P_2 \ne P_0$.
\eproof

A realizable palette $P\in\mathcal{P^+}$ is \emph{uniquely decomposable} if there exists a unique ordered pair of realizable palettes $P_1,P_2\in \mathcal{P}^+$ such that $P = P_1 \oplus P_2$. 

A realizable palette $P\in \mathcal{P}^+$ is \emph{uniquely realizable} if it is either trivial (and then the uniqueness of its realization is given by Lemma \ref{lemma:trivial}) or it is uniquely decomposable into palettes $P_1$ and $P_2$ that are both uniquely realizable. Note that since the ranks of $P_1$ and $P_2$ are strictly smaller than the rank of $P$, the unique realizability is well-defined.

\begin{lemma}
    \label{lemma:uniq-extended}
    Let $T$ be a bad tripole distinct from $T^{H_5}_1$, $T^{H_8}_1$, $T^{H_5}_2$, $T^{H_8}_2$, $T^{H_8}_3$, and $T^{H_8}_4$, or a tripole isomorphic to $T^0$, $T^1$, $T^4_1$, $T^4_2$, $T^6_1$, $T^6_2$, $T^7_1$, or $T^7_2$. Then $P(T)$ is uniquely realizable.
\end{lemma}

\beginproof The statement was checked by a computer.
\eproof

\begin{lemma}
    \label{lemma:uniq-extended-inducted}
    Let $T$ be a cubic Halin tripole such that $P(T)=P(B)$ where $B$ is a bad tripole distinct from $T^{H_5}_1$, $T^{H_8}_1$, $T^{H_5}_2$, $T^{H_8}_2$, $T^{H_8}_3$, and $T^{H_8}_4$, or a tripole isomorphic to $T^0$, $T^1$, $T^4_1$, $T^4_2$, $T^6_1$, $T^6_2$, $T^7_1$, or $T^7_2$. Then $T\cong B$.
\end{lemma}

\beginproof The statement follows directly from Lemma \ref{lemma:uniq-extended}. \eproof

{{\it Proof (of Theorem \ref{thm:total-cubic}).\ }} 
Let $H$ be a cubic Halin graph that is not totally 4-colorable. Let $v$ be a peripheral vertex of $H$. Then $T^H_v$ is a non-trivial cubic Halin tripole with an incompletable realizable palette $P$. If $P \notin \{P_A, P_B, P_C\}$, meaning it is uniquely realizable, then $P = P(B)$ for some bad tripole $B$ other than $T^{H_5}_1, T^{H_8}_1, T^{H_5}_2, T^{H_8}_2, T^{H_8}_3$, or $T^{H_8}_4$, and by Lemma \ref{lemma:uniq-extended-inducted}, $T^H_v \cong B$. If $P \in \{P_A, P_B, P_C\}$, the palette only has two realizations, both giving uniquely realizable tripoles that compose into a bad tripole. In both cases, $T^H_v$ is a bad tripole, implying that $H$ is isomorphic to $K_4$, $H_4$, $H_5$, or $H_8$. \eproof

The source codes and detailed results of the computer-assisted part can be found at github \cite{Matok-Palette-Enumerator2026}.

%% file: article/3-subcubic.tex
We can use the same methods to find similar results for subcubic Halin graphs, with respect to total coloring and also with respect to AVD coloring. Let us recall the definition of the latter.

An \emph{adjacent vertex distinguishing coloring} of a graph $G$ is a proper edge-coloring $\varphi$ such that for any two adjacent vertices $u,v\in V(G)$ we have $\varphi[u] \neq \varphi[v]$, where $\varphi[x]$ denotes the set of colors on the edges incident with the vertex $x$. 
As mentioned in the introduction, the total and AVD colorings coincide on cubic graphs when using four colors. However, if vertices of degree two are allowed, then not every AVD coloring extends to a total coloring, nor can every total 4-coloring be interpreted as an AVD coloring (simply by forgetting the colors of vertices).
Therefore, we will present the results for each coloring separately. For the AVD coloring, for a vertex $v$, instead of coloring it with a color missing at its incident edges, we consider the set of missing colors, and we impose these sets to be different for adjacent vertices.

In section \ref{sec:proof} we have built a theory describing palettes of Halin tripoles. All notions extend to subcubic Halin tripoles, with respect to both types of coloring. 
The main difference lies in the presence of vertices of degree two. More precisely,
if the root vertex of a tripole has degree two, then it only decomposes into a single smaller tripole; the palette of the former can be computed from the palette of the latter.

Let $P_0$ be the palette of the Halin tripole of rank 0 with respect to total and AVD coloring (it is the same for both and it is the same as for cubic Halin graphs). Let $\mathcal{P}^+_{total}$ and $\mathcal{P}^+_{AVD}$ be the inclusion-wise minimal sets of palettes containing $P_0$ closed under composition, with respect to the total and AVD colorings, respectively. Clearly, these are the sets of all realizable palettes with respect to each coloring. The stratification of each of these sets follows the same principle as before. Additionally, the same principle can be applied to determine their size.

\begin{lemma}
    \label{lemma:strat-sub-total}
    The size of $\mathcal{P}^+_{total}$ is 3196. Its stratification contains 16 non-empty strata.
\end{lemma}

\begin{lemma}
    \label{lemma:strat-sub-avd}
    The size of $\mathcal{P}^+_{AVD}$ is 2196. Its stratification contains 16 non-empty strata.
\end{lemma}

{{\it Proof (of Lemma \ref{lemma:strat-sub-total} and Lemma \ref{lemma:strat-sub-avd}).\ }} We adjusted the computer program described in the proof of Lemma \ref{lemma:p-plus-size} to account for root vertices of degree 2 with respect to individual colorings. \eproof

The details on respective strata sizes can be seen in Table \ref{tab:stratification-subcubic}.

The palettes corresponding to graphs $K_4, H_4, H_5$, and $H_8$ have naturally arisen in the subcubic case too. Perhaps surprisingly, $\mathcal{P}^+_{AVD}$ saw no addition of incompletable palettes compared to $\mathcal{P}^+$. Moreover, all the palettes participating in the decomposition of these incompletable palettes are uniquely realizable, but $P_A,P_B,$ and $P_C$. Fortunately, their behavior in $\mathcal{P}^+_{AVD}$ mimics its behavior in $\mathcal{P}^+$ precisely, meaning that both pairs of their respective palette decomposition already are uniquely realizable. Acknowledging that the palette of the trivial tripole is uniquely realizable in $\mathcal{P}^+_{AVD}$ (and also $\mathcal{P}^+_{total}$), we can without further a do, state the following theorem.

\begin{theorem}\label{thm:avd-subcubic}
    Let $H$ be a subcubic Halin graph other than $K_4 ,H_4, H_5$, and $H_8$. Then 
    \[
    \chi'_{AVD}(H) = 4.
    \]
\end{theorem}

{{\it Proof.\ }} Analogous to the proof of Theorem \ref{thm:total-cubic}. \eproof

The $4$-th column of Table \ref{tab:stratification-subcubic} shows that there are 9 more incompletable palettes in $\mathcal{P}^+_{total}$ than there were in $\mathcal{P}^+$. Starting at rank 2, there is a uniquely realizable palette corresponding to a subcubic Halin tripole with both spanning vertices of degree 2. This tripole corresponds to a fictional subcubic Halin graph with only two peripheral vertices, implying a multiple (peripheral) edge. At rank 3, there are 8 incompletable palettes. The graph $H_3$, depicted in Figure \ref{fig:subcubic-noncolor} gives rise to 8 subcubic Halin tripoles whose palettes of the form $T^{H_3}_v$ correspond to the 8 incompletable palettes on rank 3. We will also call these tripoles bad, extending the set of bad tripoles arising from $K_4, H_4, H_5,$ and $H_8$ described in section \ref{subsec:palettes}. As all the bad tripoles arising from $H_3$ are uniquely realizable, the claims of Lemma \ref{lemma:uniq-extended}, and inherently Lemma \ref{lemma:uniq-extended-inducted} can be extended for them as well. That again allows us to present the following theorem.

\begin{figure}[ht]
    \centering
    \includegraphics[width=0.25\linewidth]{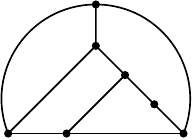}
    \caption{$H_3$ -- subcubic Halin graph with $\chi'' = 5.$}
    \label{fig:subcubic-noncolor}
\end{figure}

\begin{theorem}\label{thm:total-subcubic}
    Let $H$ be a subcubic Halin graph other than $K_4, H_4, H_5$, $H_8$, and $H_3$. Then 
    \[
    \chi''(H) = 4.
    \]
\end{theorem}

{{\it Proof.\ }} Analogous to the proof of Theorem \ref{thm:total-cubic}. \eproof

\begin{table}[ht]
    \centering
    \begin{tabular}{|c|c|c|c|c|c|c|}
        \cline{2-7}
        \multicolumn{1}{c}{ } &
        \multicolumn{3}{|c|}{$X = total$} &
        \multicolumn{3}{c|}{$X = AVD$} \\
        \hline
         Stratum & $|\mathcal{P}^+_{X}|$ & UD ($\mathcal{P}^+_{X}$) & I ($\mathcal{P}^+_{X}$) & $|\mathcal{P}^+_{X}|$ & UD ($\mathcal{P}^+_{X}$) & I ($\mathcal{P}^+_{X}$) \\ \hline
         $\mathcal{P}_0$ & 1 & 1& 1& 1 & 1 & 1\\
         $\mathcal{P}_1$ & 2 & 2& 1& 2 & 2 & 1\\
         $\mathcal{P}_2$ & 6 & 6& 1& 6 & 4 & 0\\
         $\mathcal{P}_3$ & 22 & 19& 8 & 20 & 15 & 0\\
         $\mathcal{P}_4$ & 87 & 71& 6 & 75 & 40 & 6\\
         $\mathcal{P}_5$ & 264 & 189& 7 & 177 & 94 & 7\\
         $\mathcal{P}_6$ & 523 & 331& 0 & 284 & 159 & 0\\
         $\mathcal{P}_7$ & 716 & 442& 0 & 423 & 214 & 0\\
         $\mathcal{P}_8$ & 691 & 385& 7 & 458 & 238 & 7\\
         $\mathcal{P}_9$ & 446 & 278& 0 & 309 & 159 & 0\\
         $\mathcal{P}_{10}$ & 274 & 164& 0 & 222 & 131 & 0\\
         $\mathcal{P}_{11}$ & 105 & 70& 0 & 140 & 72 & 0\\
         $\mathcal{P}_{12}$ & 43 & 28& 0 & 57 & 39 & 0\\
         $\mathcal{P}_{13}$ & 8 & 6& 0 & 12 & 4 & 0\\
         $\mathcal{P}_{14}$ & 6 & 6& 0 & 4 & 2 & 0\\
         $\mathcal{P}_{15}$ & 2 & 2& 0 & 6 & 6 & 0\\ \hline
         $\mathcal{P}_{i}, i \geq 16$ & 0 & 0 & 0 & 0 & 0 & 0 \\ \hline
         total & 3196 & 2000 & 31 & 2196 & 1180 & 22\\ \hline
    \end{tabular}
    \caption{The sizes of strata of  $\mathcal{P}^+_{total}$ and $\mathcal{P}^+_{AVD}$. For each type of coloring, the number of all, uniquely decomposable (UD), and incompletable (I) palettes are given.}
    \label{tab:stratification-subcubic}
\end{table}

%% file: article/4-conclusion.tex
In this article, we have fully characterized cubic and subcubic Halin graphs which are not totally or AVD 4-colorable. As an auxiliary result we demonstrated a linear-time algorithm yielding instances of total (or AVD) coloring in a (sub)cubic Halin graph.

As a final remark, we would like to interest the reader in yet another, very close in nature, coloring. The strict neighbor distinguishing (SND) coloring, introduced by Przybyło and Kwaśny \cite{Przybylo20215} and independently by Gu et al. \cite{Gu2021355}, puts even stronger constraints on the chromatic neighborhoods of adjacent vertices. The formal definition is as follows. A \emph{strict neighbor distinguishing coloring} of a graph $G$ is a proper edge-coloring $\varphi$ with $k$ colors, such that for any two adjacent vertices $u,v\in V(G)$ : $\varphi[u] \not\subseteq \varphi[v]$ and $\varphi[v] \not\subseteq \varphi[u]$. The smallest $k$ for which $G$ admits a SND-coloring is denoted by $\chi'_{snd} (G)$. In cubic graphs, the problem is equivalent to AVD-coloring as $\varphi[u] \subseteq \varphi[v]$ only if $\varphi[u] = \varphi[v]$, and therefore also to total coloring. For subcubic graphs, however, this argument does not hold. It is known that $\chi'_{snd}(G) \leq 7$ for any subcubic graph $G$ \cite{Gu2021355}, moreover, the class of subcubic graphs attaining this bound consists of a single graph (which is not Halin).

In Figure \ref{fig:chair-config} we demonstrate that the class of subcubic Halin graphs with $\chi'_{snd}(G) > 4$ is infinite, in contrast to total and AVD- colorings. We identified a tripole whose SND-palette is empty when using 4 colors, hence, any Halin graph containing it is inherently not SND-4-colorable. For details, see Figure \ref{fig:chair-config}. 

\begin{figure}
    \centering
    \includegraphics{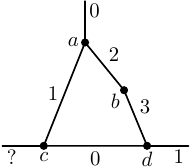}
    \caption{A tripole with an empty SND-palette. Assume it has an SND coloring $\varphi$ using only four colors. 
    Without loss of generality, let the colors at $a$ be 0, 1, 2 as in the drawing. Then $\varphi(bd)=3$, otherwise $\varphi[b]\subset \varphi[a]$.
    To avoid $\varphi[b]\subset \varphi[d]$, the color 2 cannot be used at $d$, and so $\varphi(cd)=0$. But then either $\varphi[c]=\varphi[a]$ or $\varphi[c]=\varphi[d]$, a contradiction.}
    \label{fig:chair-config}
\end{figure}